\documentclass[11pt]{amsart}

\usepackage{amsthm}
\usepackage{amsmath}
\usepackage{amssymb}

\newtheorem{thm}{Theorem}[section]
\newtheorem{theorem}[thm]{Theorem}

\theoremstyle{remark}
\newtheorem{remark}[thm]{Remark}

\newcommand{\RR}{\mathbb R}

\newcommand{\cH}{\mathbb H}

\title{Algorithmic constructions of unitary matrices and tight frames}
\author[J.C. Tremain ]{Janet C. Tremain }
\address{Department of Mathematics, University
of Missouri, Columbia, MO 65211-4100}

\thanks{The author is supported by DTRA/NSF DMS 1042701}

\email{j.tremain@mchsi.com}

\begin{document}

\maketitle

\begin{abstract}
We give a number of algorithms for constructing unitary matrices and tight
frames with specialized properties.  These were produced at the request
of researchers at the Frame Research Center\\  (www.framerc.org) to help with
their research on fusion frames, the Kadison-Singer Problem and equiangular
tight frames.
\end{abstract}

\section{Introduction}

We will give a sequence of algorithms for constructing unitary matrices
(i.e. orthonormal bases) and tight frames with specialized properties.
For unitary matrices, we just need to construct square matrices with
rows and columns orthogonal and row and column numbers square summing
to one.  For tight frames, we need to produce $M\times N$ matrices with
$M\ge N$ which have orthogonal columns and whose column numbers square
sum to a fixed constant.  Then the row vectors of this matrix will form a tight frame
for $\cH_N$.  We will be interested in a number of special properties such as
sparse orthonormal bases for hyperplanes,
matrices with first row a constant, first two rows a constant, matrices made from
just 2 constants, matrices built out of unitary matrices, matrices with a large portion
of the norms of their vectors allocated to special positions in the matrix etc.
We will work specifically with 2-tight frames since it is known that {\it paving}
these is equivalent to solving the famous Kadison-Singer Problem \cite{CE,CT,CT1}.

Recall that a family of vectors $\{f_i\}_{i=1}^M$ is a {\it frame} for a Hilbert space
$\cH_N$ if there are constants $0<A\le B <\infty$, called the {\it lower and upper
frame bounds} respectively so that
\[ A\|f\|^2 \le \sum_{i=1}^M |\langle f,f_i\rangle |^2 \le B \|f\|^2.\]
The frame is {\it A-tight} if $A=B$ and {\it Parseval} if $A=B=1$.  It is equal norm
of $\|f_i\|=c$, for all $i=1,2,\ldots,M$, and {\it unit norm} if $c=1$.  The {\it analysis
operator} of the frame is $T:\cH_N \rightarrow \ell_2^M$ given by
\[ Tf = \sum_{i=1}^M \langle f,f_i\rangle e_i,\]
where $\{e_i\}_{i=1}^M$ is the natural orthonormal basis of $\ell_2^M$.  The 
adjoint of the analysis operator is the {\it synthesis operator} which satisfies
\[ T^*\left ( \{a_i\}_{i=1}^M \right ) = \sum_{i=1}^M a_if_i.\]
The frame operator is the positive, self-adjoint invertible operator given by
\[ S =: T^*Tf = \sum_{i=1}^M \langle f,f_i\rangle f_i.\]
The synthesis operator
of the frame is the matrix with the frame vectors as column vectors:
\[ A = \begin{bmatrix}
|&&|&|&|\\
f_1&\cdots&f_2&\cdots&f_M\\	
|&&|&|&|
\end{bmatrix}
\] 

\section{Sparse Orthonormal Bases for Hyperplanes}

For their work on {\it sparsity} the researchers at the Frame Research Center needed
the sparsest orthonormal bases for hyperplanes in $\RR^N$ which do not contain
the span of any subset of the orthonormal basis.  We believe that this class of hyperplanes
is given by the following cases - although at this time we do not have a proof that these
are the sparsest.

\vskip12pt
\noindent {\bf dimension n=2:}
 \[
A_2 =  \begin{bmatrix}
 1&1
 \end{bmatrix} \]
 \vskip12pt
 
 \noindent {\bf dimension n=3:}
 \[ \begin{bmatrix}
 1&2&1\\
 0& -1&2
 \end{bmatrix} 
 \]
 \vskip12pt
\noindent {\bf dimension n=4:}

\[ A_4 = \begin{bmatrix}
1&1&1&1\\
1&-1&0&0\\
0&0&1&-1
\end{bmatrix}
\]
The above has sparsity $8$ and it is easily checked that any orthonormal basis for
this hyperplane in $\RR^4$ has sparsity $\ge 8$.
\vskip12pt
\noindent {\bf dimension n=5:}  Here we combine the earlier cases of $A_2$ and 
$A_3$ and then add a non-zero row on top.
 \[ A_5=\begin{bmatrix}
 *&*&*&*&*\\
 *&*&*&0&0\\
 0&*&*&0&0&\\
 0&0&0&*&*
 \end{bmatrix} \]
The above has sparsity 11.

\vskip12pt
\noindent {\bf dimension n=6:}  Here, we will combine the cases $2n=2$ with $2n=4$
and then insert a new top row.
\[ A_6 = \begin{bmatrix}
1&1&1&1&1&1\\
1&1&-1&-1&0&0\\
1&-1&0&0&0&0\\
0&0&1&-1&0&0\\
0&0&0&0& 1&-1
\end{bmatrix}
\]
The above has sparsity 16.
\vskip12pt
\noindent {dimension $n=7$:}  Here we combine the earlier cases of
$n=4$ and $n=3$.
\vskip12pt
\[ A_7=
 \begin{bmatrix}
 *&*&*&*&*&*&*\\
 *&*&*&*&0&0&0\\
 *&*&0&0&0&0&0\\
 0&0&*&*&0&0&0\\
 0&0&0&0&*&*&*\\
 0&0&0&0&0&*&*
 \end{bmatrix} \]
This has sparsity 20. 

\vskip12pt
\noindent {\bf dimension n=8:}  In this case we combine two of the cases $2n=4$ and
add a new top row.
 \[ A_8=\begin{bmatrix}
 1&1&1&1&1&1&1&1\\
 1&1&-1&-1&0&0&0&0\\
 1&-1&0&0&0&0&0&0\\
 0&0&1&-1&0&0&0&0&\\
 0&0&0&0&1&1&-1&-1\\
 0&0&0&0& 1&-1&0&0&\\
 0&0&0&0&0&0&1&-1
 \end{bmatrix}
 \]
This has sparsity 24.
\vskip12pt
The general cases here are:
\vskip12pt
\noindent {\bf dimension 2n with n even:}
We combine two of the cases $A_{2n}$ and add a new top row.
\[ A_{4n}=\begin{bmatrix}
2n\ ones & 2n\ ones\\
A_{2n}& A_{2n}
\end{bmatrix} 
\]
This has sparsity:  
\[ 2\mbox{[sparsity of }A_{2n}] + 4n.\]

\noindent {\bf dimension 2(n+1), n+1 odd:}
We combine the previous cases n+2 and n.
\[ \begin{bmatrix}
n+2 \ ones & n \ ones\\
A_{n+2}& A_{n}
\end{bmatrix}\]
This has sparsity
\[ \mbox{sparsity of }A_{n+2} + \mbox{ sparsity of }A_{n} + 2n+2.\]

\vskip12pt
\noindent {\bf dimension $2n+1=2^k+1+2i$ with i even:}  Here we combine the cases
$n=2^{k-1}+2i$ and $2^{k-1}+1$ and add a new top row.
\[ A_{2n+1}=\begin{bmatrix}
2^{k-1}+2i\ stars & 2^{k-1}+1\ stars\\
A_{2^{k-1}+2i}& A_{2^{k-1}+1}
\end{bmatrix} 
\]

\vskip12pt

\noindent {\bf dimension $2n+1=2^k+1+2i$ with i odd:}
\vskip12pt
\vskip12pt
\[ A_{2n+1}=\begin{bmatrix}
2^{k-1}\ stars & 2^{k-1}+1+2i\ stars\\
A_{2^{k-1}+2i}& A_{2^{k-1}+1+2i}
\end{bmatrix} 
\]

\section{Unitary Matrices with  First Row a Constant}

The researchers at the Frame Research Center needed unitary matrices with the
first row a constant and maximal sparsity for their work on equiangular fusion frames.
They also wanted a large number of zero's in the matrix.
For this algorithm, we can just alter the top row, then add a new top row and first column to the
previous example.

To make this more clear, we do two small examples and then the general case.  First we do
$3 \times 3$.

\[ A_3 = \begin{bmatrix}
\frac{1}{\sqrt{3}}& \frac{1}{\sqrt{3}}& \frac{1}{\sqrt{3}}\\
 -\frac{2}{\sqrt{3\cdot 2}} & \frac{1}{\sqrt{3\cdot 2}}  & \frac{1}{\sqrt{3\cdot 2}} \\
0 & - \frac{1}{\sqrt{2}}& \frac{1}{\sqrt{2}}
\end{bmatrix}\]

Next we do $4\times 4$.

\[ A_4= \begin{bmatrix}
\frac{1}{\sqrt{4}}&\frac{1}{\sqrt{4}}&\frac{1}{\sqrt{4}}&\frac{1}{\sqrt{4}}\\
-\frac{3}{\sqrt{4 \cdot 3}}  
& \frac{1}{\sqrt{4 \cdot 3}}  & \frac{1}{\sqrt{4 \cdot 3}}  & \frac{1}{\sqrt{4 \cdot 3}}\\
0& -\frac{2}{\sqrt{3\cdot 2}} & \frac{1}{\sqrt{3\cdot 2}}  & \frac{1}{\sqrt{3\cdot 2}} \\
0& 0 & - \frac{1}{\sqrt{2}}& \frac{1}{\sqrt{2}}
\end{bmatrix}\]

Now for the general $n\times n$ case.

\[ A_n = \begin{bmatrix}
\frac{1}{\sqrt{n}}& \cdots & \frac{1}{\sqrt{n}}& \frac{1}{\sqrt{n}}& \frac{1}{\sqrt{n}} & \frac{1}{\sqrt{n}}
& \frac{1}{\sqrt{n}}\\
- \frac{n-1}{\sqrt{n(n-1)}}& \cdots & \frac{1}{\sqrt{n(n-1)}}& \frac{1}{\sqrt{n(n-1)}}&
 \frac{1}{\sqrt{n(n-1)}}& \frac{1}{\sqrt{n(n-1)}}& 
\frac{1}{\sqrt{n(n-1)}}\\
\vdots & \vdots & \vdots & \vdots & \vdots & \vdots & \vdots\\
0 & \cdots & -\frac{4}{\sqrt{5 \cdot 4}}& \frac{1}{\sqrt{5 \cdot 4}}&
 \frac{1}{\sqrt{5 \cdot 4}}& \frac{1}{\sqrt{5 \cdot 4}}&
\frac{1}{\sqrt{5 \cdot 4}}\\
0 & \cdots &0 & -\frac{3}{\sqrt{4 \cdot 3}}  
& \frac{1}{\sqrt{4 \cdot 3}}  & \frac{1}{\sqrt{4 \cdot 3}}  & \frac{1}{\sqrt{4 \cdot 3}}\\
0& \cdots & 0 &0& -\frac{2}{\sqrt{3\cdot 2}} & \frac{1}{\sqrt{3\cdot 2}}  & \frac{1}{\sqrt{3\cdot 2}}   \\
0 & \cdots & 0 & 0 & 0 & - \frac{1}{\sqrt{2}}& \frac{1}{\sqrt{2}}
\end{bmatrix}\]

\section{Unitary Matrices with First Two Rows a Constant Modulus}

The researchers at the FRC needed these matrices for their work on constructing
tight fusion frames.

Again, we give a few small cases to get this started.
\vskip12pt
\noindent {\bf dimension 4:}
\vskip12pt
\[ A_4 = \begin{bmatrix}
 -1&1&1&1\\
1&-1&1&1\\
1&1&-1&1\\
1&1&1&-1
\end{bmatrix}\]
\vskip12pt
\noindent {\bf dimension 6:}
\vskip12pt
\[ A_6 = \begin{bmatrix}
\sqrt{\frac{1}{6}}&\sqrt{\frac{1}{6}}&\sqrt{\frac{1}{6}}&
\sqrt{\frac{1}{6}}&\sqrt{\frac{1}{6}}&\sqrt{\frac{1}{6}}\\
-\sqrt{\frac{1}{6}}&-\sqrt{\frac{1}{6}}&-\sqrt{\frac{1}{6}}&
\sqrt{\frac{1}{6}}&\sqrt{\frac{1}{6}}&\sqrt{\frac{1}{6}}\\
-\sqrt{\frac{2}{3}}&\sqrt{\frac{1}{6}}&\sqrt{\frac{1}{6}}&
0&0&0\\
0& -\sqrt{\frac{1}{2}}&\sqrt{\frac{1}{2}}&0&0&0\\
0&0&0&-\sqrt{\frac{2}{3}}&\sqrt{\frac{1}{6}}& \sqrt{\frac{1}{6}}\\
0&0&0&0& -\sqrt{\frac{1}{2}}&\sqrt{\frac{1}{2}}
\end{bmatrix}
\]

\vskip12pt
For $A_8$ we can use a Hadamard again.  But I was asked for "sparse"
matrices doing this.
\vskip12pt
\noindent {\bf dimension 8:}
\vskip12pt
\[ A_8 = \begin{bmatrix}
\sqrt{\frac{1}{8}}&\sqrt{\frac{1}{8}}&\sqrt{\frac{1}{8}}&
\sqrt{\frac{1}{8}}&
\sqrt{\frac{1}{8}}&\sqrt{\frac{1}{8}}&\sqrt{\frac{1}{8}}&
\sqrt{\frac{1}{8}}\\
-\sqrt{\frac{1}{8}}&-\sqrt{\frac{1}{8}}&-\sqrt{\frac{1}{8}}&
-\sqrt{\frac{1}{8}}&
\sqrt{\frac{1}{8}}&\sqrt{\frac{1}{8}}&\sqrt{\frac{1}{8}}&
\sqrt{\frac{1}{8}}\\
-\sqrt{\frac{3}{4}}&\sqrt{\frac{1}{12}}&\sqrt{\frac{1}{12}}&
\sqrt{\frac{1}{12}}&0&0&0&0\\
0& -\sqrt{\frac{2}{3}}&\sqrt{\frac{1}{6}}& \sqrt{\frac{1}{6}}&0&0&0&0\\
0&0&-\sqrt{\frac{1}{2}}& \sqrt{\frac{1}{2}}&0&0&0&0\\
0&0&0&0&-\sqrt{\frac{3}{4}}&\sqrt{\frac{1}{12}}&\sqrt{\frac{1}{12}}&
\sqrt{\frac{1}{12}}\\
0&0&0&0&0& -\sqrt{\frac{2}{3}}&\sqrt{\frac{1}{6}}& \sqrt{\frac{1}{6}}\\
0&0&0&0&0&0&-\sqrt{\frac{1}{2}}& \sqrt{\frac{1}{2}}
\end{bmatrix} \]
\vskip12pt
And in general:
\vskip12pt
\noindent {\bf $A_{2n}$=}
\vskip12pt
\[  \begin{bmatrix}
\sqrt{\frac{1}{2n}}& \sqrt{\frac{1}{2n}}&\cdots & \sqrt{\frac{1}{2n}}&\sqrt{\frac{1}{2n}}&
\sqrt{\frac{1}{2n}}&\cdots &\sqrt{\frac{1}{2n}}\\
-\sqrt{\frac{1}{2n}}& -\sqrt{\frac{1}{2n}}&\cdots &- \sqrt{\frac{1}{2n}}&\sqrt{\frac{1}{2n}}&
\sqrt{\frac{1}{2n}}&\cdots &\sqrt{\frac{1}{2n}}\\
-\sqrt{\frac{n-1}{n}}&\sqrt{\frac{1}{n(n-1)}}&\cdots & \sqrt{\frac{1}{n(n-1)}}&
0&0&\cdots &0\\
\vdots&\vdots &\vdots&\vdots&\vdots&\vdots&\vdots&\vdots\\
0&0&\cdots & \sqrt{\frac{1}{2}}&0&0&\cdots &0\\
0&0&\cdots &0&-\sqrt{\frac{n-1}{n}}&\sqrt{\frac{1}{n(n-1)}}&\cdots & \sqrt{\frac{1}{n(n-1)}}&\\
\vdots&\vdots &\vdots&\vdots&\vdots&\vdots&\vdots&\vdots\\
0&0&\cdots&0&0&0&\cdots &\sqrt{\frac{1}{2}}
\end{bmatrix}\]

\section{Unitary Matrices:  Made from 2 Constants}
We will look at various ways to construct unitary matrices from
just two constants.

\subsection{First Case}
In this section we construct a very special class of unitary of the
form

\[ \frac{1}{n}\begin{bmatrix}
-b&a&a&a&\cdots &a\\
a&-b&a&a&\cdots &a\\
a&a&-b&a& \cdots &a\\
\vdots&\vdots&\vdots&\vdots&\vdots&\vdots\\
a&a&a&a& \cdots &-b
\end{bmatrix}\]

For the case of $\RR^3$ this looks like:
\[ \frac{1}{3}\begin{bmatrix}
-1&2&2\\
2&-1&2\\
2&2&-1
\end{bmatrix}
\]

For the case $\RR^6$ this looks like:

\[ \frac{1}{3}\begin{bmatrix}
-2&1&1&1&1&1\\
1&-2&1&1&1&1\\
1&1&-2&1&1&1\\
1&1&1&-2&1&1\\
1&1&1&1&-2&1\\
1&1&1&1&1&-2
\end{bmatrix}\]

Or in a general form, the above looks like:

\[ \frac{1}{6}\begin{bmatrix}
-4&2&2&2&2&2\\
2&-4&2&2&2&2\\
2&2&-4&2&2&2\\
2&2&2&-4&2&2\\
2&2&2&2&-4&2\\
2&2&2&2&2&-4
\end{bmatrix}\]

And for the general case of $\RR^n$ we have:

\[ \frac{1}{n}\begin{bmatrix}
-(n-2)&2&2&\cdots &2\\
2&-(n-2)&2&\cdots &2\\
2&2&-(n-2)&\cdots &2\\
\vdots&\vdots&\vdots&\vdots&\vdots\\
2&2&2&\cdots & -(n-2)
\end{bmatrix}\]

\subsection{Second Case}

In this section we construct a class of unitary matrices of the form

\[ \begin{bmatrix}
-a&a&a&a&-b&b&b&b\\
a&-a&a&a&b&-b&b&b\\
a&a&-a&a&b&b&-b&b\\
a&a&a&-a&b&b&b&-b\\
-b&b&b&b&a&-a&-a&-a\\
b&-b&b&b&-a&a&-a&-a\\
b&b&-b&b&-a&-a&a&-a\\
b&b&b&-b&-a&-a&-a&a
\end{bmatrix}\]

The rows and columns of this matrix are clearly orthogonal.  To make the
rows and columns square sum to one, we just need to have that
$4a^2+4b^2=1$.
For example, in the above case, $b=\frac{1}{\sqrt{20}}$ and $a=2b$ work.  
The above matrix represents a generalization of the Hadamard construction.
In general it looks like:

\begin{theorem}
Let $A,B$ be $n\times n$ unitary matrices.  Choose $a,b$ so that
$a^2+b^2=1$.  Then the following matrix is unitary:

\[C = \begin{bmatrix}
aA&bB\\
aA&-bB
\end{bmatrix}\]
\end{theorem}

\begin{proof}
A direct calculation gives that $C^*C=I$.
\end{proof}

More general than this is:

\begin{theorem}
Let $A_1,A_2,\ldots,A_n$ be $m\times m$ unitary matrices and let
\[ \left ( a_{ij} \right )_{i,j=1}^n,\mbox{ be a unitary matrix.}\]
Then the following matrix is a $nm\times nm$ unitary matrix:
\[ B=\begin{bmatrix}
a_{11}A_1& a_{12}A_2 & \cdots &a_{1n}A_n\\
a_{21}A_1&a_{22}A_2& \cdots &a_{2n}A_n\\
\vdots & \vdots& \vdots& \vdots\\
a_{n1}A_1&a_{2n}A_2 & \cdots & a_{nn}A_n
\end{bmatrix}\]
\end{theorem}

\begin{proof}
A direct calculation gives that $B^*B=I$.
\end{proof}

\begin{remark}The above theorem can be generalized to tight frames - which we
do in a later section.
\end{remark}

\subsection{Third Case}

In this section we find unitary matrices of the form

\[ \begin{bmatrix}
a&a&a&-b&-b&b\\
a&a&a&-b&b&-b\\
a&a&a&b&-b&-b\\
-b&b&b&a&a&a\\
b&-b&b&a&a&a\\
b&b&-b&a&a&a
\end{bmatrix}\]

For this to work, we need:
\[ 3a^2+3b^2=1,\]
and for the inner products to be zero,
\[ 3a^2-b^2=0.\]
Solving we get:
\[ b=\pm \frac{1}{2},\mbox{ and } a = \pm \sqrt{\frac{1}{12}}.\]

\section{Two Tight Frames}

Here we will give a general construction for unit norm 2-tight frames.   The FRC
needed such examples made up of a single vector. 
Unit norm 2-tight frames
are the row vectors of a $2n\times n$ matrix having the following properties.
\vskip10pt

(1) The rows square sum to one.
\vskip10pt

(2)  The columns square sum to two.
\vskip10pt

(3)  The column vectors are orthogonal.
\vskip12pt

We start with a concrete
example and then go to the general case.

\[ \begin{bmatrix}
- \sqrt{\frac{7}{16}}& \sqrt{\frac{5}{16}}& \sqrt{\frac{3}{16}}& \sqrt{\frac{1}{16}}\\
\sqrt{\frac{7}{16}}& -\sqrt{\frac{5}{16}}& \sqrt{\frac{3}{16}}& \sqrt{\frac{1}{16}}\\
\sqrt{\frac{7}{16}}& \sqrt{\frac{5}{16}}& -\sqrt{\frac{3}{16}}& \sqrt{\frac{1}{16}}\\
\sqrt{\frac{7}{16}}& \sqrt{\frac{5}{16}}& \sqrt{\frac{3}{16}}& -\sqrt{\frac{1}{16}}\\
-\sqrt{\frac{1}{16}}& \sqrt{\frac{3}{16}}& \sqrt{\frac{5}{16}}&\sqrt{\frac{7}{16}}\\
\sqrt{\frac{1}{16}}&- \sqrt{\frac{3}{16}}& \sqrt{\frac{5}{16}}&\sqrt{\frac{7}{16}}\\
\sqrt{\frac{1}{16}}& \sqrt{\frac{3}{16}}&- \sqrt{\frac{5}{16}}&\sqrt{\frac{7}{16}}\\
\sqrt{\frac{1}{16}}& \sqrt{\frac{3}{16}}& \sqrt{\frac{5}{16}}&-\sqrt{\frac{7}{16}}\\
\end{bmatrix}\]

If we let
\[ A = \begin{bmatrix} 
- \sqrt{\frac{7}{16}}& \sqrt{\frac{5}{16}}& \sqrt{\frac{3}{16}}& \sqrt{\frac{1}{16}}\\
\sqrt{\frac{7}{16}}& -\sqrt{\frac{5}{16}}& \sqrt{\frac{3}{16}}& \sqrt{\frac{1}{16}}\\
\sqrt{\frac{7}{16}}& \sqrt{\frac{5}{16}}& -\sqrt{\frac{3}{16}}& \sqrt{\frac{1}{16}}\\
\sqrt{\frac{7}{16}}& \sqrt{\frac{5}{16}}& \sqrt{\frac{3}{16}}& -\sqrt{\frac{1}{16}}\\
\end{bmatrix}\]

We can think of reversing the rows of $A$ to get the matrix:

\[ 
A^R = \frac{1}{\sqrt{16}}
\begin{bmatrix}
\sqrt{7}&\sqrt{5}&\sqrt{3}&-\sqrt{1}\\
\sqrt{7}&\sqrt{5}&-\sqrt{3}&\sqrt{1}\\
\sqrt{7}&-\sqrt{5}&\sqrt{3}&\sqrt{1}\\
-\sqrt{7}&\sqrt{5}&\sqrt{3}&\sqrt{1}

\end{bmatrix}\]

We can also think of reversing the columns of $A$ to get:
\[ A^C = \frac{1}{16} \begin{bmatrix}
\sqrt{1}&\sqrt{3}&\sqrt{3}&-\sqrt{7}\\
\sqrt{1}&\sqrt{3}&-\sqrt{5}&\sqrt{7}\\
\sqrt{1}&-\sqrt{3}&\sqrt{5}&\sqrt{7}\\
-\sqrt{1}&\sqrt{3}&\sqrt{5}& \sqrt{7}
\end{bmatrix}\]

If we reverse the rows of $A$ and then reverse the columns of the resulting
matrix we get:

\[ A^{RC}= \begin{bmatrix}
-\sqrt{\frac{1}{16}}& \sqrt{\frac{3}{16}}& \sqrt{\frac{5}{16}}&\sqrt{\frac{7}{16}}\\
\sqrt{\frac{1}{16}}&- \sqrt{\frac{3}{16}}& \sqrt{\frac{5}{16}}&\sqrt{\frac{7}{16}}\\
\sqrt{\frac{1}{16}}& \sqrt{\frac{3}{16}}&- \sqrt{\frac{5}{16}}&\sqrt{\frac{7}{16}}\\
\sqrt{\frac{1}{16}}& \sqrt{\frac{3}{16}}& \sqrt{\frac{5}{16}}&-\sqrt{\frac{7}{16}}\\
\end{bmatrix}\]

So our first matrix is of the form:

\[ \begin{bmatrix}
A\\
A^{RC}
\end{bmatrix}\]

Let us do this one more time.
\[
\frac{1}{\sqrt{64}}\begin{bmatrix}
-\sqrt{15}& \sqrt{13}& \sqrt{11}& \sqrt{9} &-\sqrt{7}& \sqrt{5}& \sqrt{3}& \sqrt{1}\\
\sqrt{15}& -\sqrt{13}& \sqrt{11}& \sqrt{9} &\sqrt{7}& -\sqrt{5}& \sqrt{3}& \sqrt{1}\\
\sqrt{15}& \sqrt{13}&- \sqrt{11}& \sqrt{9} &\sqrt{7}& \sqrt{5}&- \sqrt{3}& \sqrt{1}\\
\sqrt{15}& \sqrt{13}& \sqrt{11}&- \sqrt{9} &\sqrt{7}& \sqrt{5}& \sqrt{3}&- \sqrt{1}\\
-\sqrt{15}& \sqrt{13}& \sqrt{11} &\sqrt{9} &\sqrt{7}&- \sqrt{5}&- \sqrt{3}&- \sqrt{1}\\
\sqrt{15}&- \sqrt{13}& \sqrt{11} &\sqrt{9} &-\sqrt{7}& \sqrt{5}&- \sqrt{3}& -\sqrt{1}\\
\sqrt{15}& \sqrt{13}&- \sqrt{11}& \sqrt{9} &-\sqrt{7}&- \sqrt{5}&\sqrt{3}& -\sqrt{1}\\
\sqrt{15}& \sqrt{13}& \sqrt{11}&- \sqrt{9} &-\sqrt{7}& -\sqrt{5}& -\sqrt{3}& \sqrt{1}\\
-\sqrt{1}& \sqrt{3}& \sqrt{5}& \sqrt{7} &-\sqrt{9}& \sqrt{11}& \sqrt{13}& \sqrt{15}\\
\sqrt{1}& -\sqrt{3}& \sqrt{5}& \sqrt{7} &\sqrt{9}&- \sqrt{11}& \sqrt{13}& \sqrt{15}\\
\sqrt{1}& \sqrt{3}& -\sqrt{5}& \sqrt{7} &\sqrt{9}& \sqrt{11}&- \sqrt{13}& \sqrt{15}\\
\sqrt{1}& \sqrt{3}& \sqrt{5}& -\sqrt{7} &\sqrt{9}& \sqrt{11}& \sqrt{13}&- \sqrt{15}\\
-\sqrt{1}& \sqrt{3}& \sqrt{5}& \sqrt{7} &\sqrt{9}&- \sqrt{11}& -\sqrt{13}& -\sqrt{15}\\
\sqrt{1}&- \sqrt{3}& \sqrt{5}& \sqrt{7} &-\sqrt{9}& \sqrt{11}&- \sqrt{13}&- \sqrt{15}\\
\sqrt{1}& \sqrt{3}&- \sqrt{5}& \sqrt{7} &-\sqrt{9}& -\sqrt{11}& \sqrt{13}&- \sqrt{15}\\
\sqrt{1}& \sqrt{3}& \sqrt{5}&- \sqrt{7} &-\sqrt{9}& -\sqrt{11}&- \sqrt{13}& \sqrt{15}\\
\end{bmatrix}\]

In our earlier notation, this is of the form:

\[ \begin{bmatrix}
A&B\\
A&-B\\
B^{RC}&A^{RC}\\
B^{RC}& -A^{RC}
\end{bmatrix}\]

This construction will always work as long as the top row is
an arithmetic progression.  So the general case here looks like (in the $8\times 4$ case):

\[ \frac{1}{\sqrt{4a+6b}}
\begin{bmatrix}
-\sqrt{a+3b}&\sqrt{a+2b}&\sqrt{a+b}&\sqrt{a}\\
\sqrt{a+3b}&-\sqrt{a+2b}&\sqrt{a+b}&\sqrt{a}\\
\sqrt{a+3b}&\sqrt{a+2b}&-\sqrt{a+b}&\sqrt{a}\\
\sqrt{a+3b}&\sqrt{a+2b}&\sqrt{a+b}&\sqrt{a}\\
-\sqrt{a}&\sqrt{a+b}&\sqrt{a+2b}&-\sqrt{a+3b}\\
\sqrt{a}&-\sqrt{a+b}&\sqrt{a+2b}&\sqrt{a+3b}\\
\sqrt{a}&\sqrt{a+b}&-\sqrt{a+2b}&\sqrt{a+3b}\\
\sqrt{a}&\sqrt{a+b}&\sqrt{a+2b}&-\sqrt{a+3b}\\
\end{bmatrix}\]
The columns are orthogonal, the rows square sum to 1 and the columns square sum to 2.
i.e.  This is a 2-tight frame.
We can then iterate this procedure as above to get ever larger examples.

\section{Two Tight Frames:  Weight in Front}
The idea now is to create 2-tight frames which have a big portion of their
weight at the beginning.  A group at the FRC wanted this for
their work on the Kadison-Singer Problem.  This will give an alternative concrete example
of non-2-pavable projections.  In \cite{CE} it was first shown that the class of projections
with constant diagonal $1/2$ are not two-pavable.  Later, in \cite{CF} a concrete
construction of such projections was given using variations of the discrete Fourier
transform.

Let us start with an example and then see that it is a very general concept.

\[ \frac{1}{\sqrt{64}}\begin{bmatrix}
-\sqrt{15}& \sqrt{14}&\sqrt{13}& \sqrt{12} & - \sqrt{4}&\sqrt{3}&\sqrt{2}&\sqrt{1}\\
\sqrt{15}&- \sqrt{14}&\sqrt{13}& \sqrt{12} & \sqrt{4}&-\sqrt{3}&\sqrt{2}&\sqrt{1}\\
\sqrt{15}& \sqrt{14}&-\sqrt{13}& \sqrt{12} & \sqrt{4}&\sqrt{3}&-\sqrt{2}&\sqrt{1}\\
\sqrt{15}& \sqrt{14}&\sqrt{13}&- \sqrt{12} & \sqrt{4}&\sqrt{3}&\sqrt{2}&-\sqrt{1}\\
-\sqrt{15}& \sqrt{14}&\sqrt{13}& \sqrt{12} & \sqrt{4}&-\sqrt{3}&-\sqrt{2}&-\sqrt{1}\\
\sqrt{15}& -\sqrt{14}&\sqrt{13}& \sqrt{12} &  -\sqrt{4}&\sqrt{3}&-\sqrt{2}&-\sqrt{1}\\
\sqrt{15}& \sqrt{14}&-\sqrt{13}& \sqrt{12} & - \sqrt{4}&-\sqrt{3}&\sqrt{2}&-\sqrt{1}\\
\sqrt{15}& \sqrt{14}&\sqrt{13}&- \sqrt{12} & - \sqrt{4}&-\sqrt{3}&-\sqrt{2}&\sqrt{1}\\
-\sqrt{1}&\sqrt{2}&\sqrt{3}& \sqrt{4}& -\sqrt{12}& \sqrt{13}& \sqrt{14}&\sqrt{15}\\
\sqrt{1}&-\sqrt{2}&\sqrt{3}& \sqrt{4}& \sqrt{12}&- \sqrt{13}& \sqrt{14}&\sqrt{15}\\
\sqrt{1}&\sqrt{2}&-\sqrt{3}& \sqrt{4}& \sqrt{12}& \sqrt{13}&- \sqrt{14}&\sqrt{15}\\
\sqrt{1}&\sqrt{2}&\sqrt{3}&- \sqrt{4}& \sqrt{12}& \sqrt{13}& \sqrt{14}&-\sqrt{15}\\
-\sqrt{1}&\sqrt{2}&\sqrt{3}& \sqrt{4}& \sqrt{12}&- \sqrt{13}&- \sqrt{14}&-\sqrt{15}\\
\sqrt{1}&-\sqrt{2}&\sqrt{3}& \sqrt{4}& -\sqrt{12}& \sqrt{13}& -\sqrt{14}&-\sqrt{15}\\
\sqrt{1}&\sqrt{2}&-\sqrt{3}& \sqrt{4}& -\sqrt{12}&- \sqrt{13}& \sqrt{14}&-\sqrt{15}\\
\sqrt{1}&\sqrt{2}&\sqrt{3}&- \sqrt{4}& -\sqrt{12}&- \sqrt{13}&- \sqrt{14}&\sqrt{15}\\

\end{bmatrix}\]
We note that in the matrix above, the columns are orthogonal and the columns
square sum to 2 and the rows square sum to 1.

Now we will seriously generalize this example by understanding why it is
working.   What makes this work comes from the proof of how we get the
formula for $1+2+3+\cdots +n$.   To do this, we write these numbers in
reverse and add them in pairs
\[ \begin{bmatrix}
1&2&3&\cdots & n-2 & n-1 & n\\
n & n-1& n-2 & \cdots & 3&2&1
\end{bmatrix}\] 

We note that the sum of the columns are all the same and they add up to twice the
sum of the rows.  So we will work with integers $n_1 > n_2 >\ldots >n_{2^{k-1}}$
and decide what we want the sums to equal, say m, and consider the matrix

\[ \begin{bmatrix}
\sqrt{n_1}&\sqrt{n_2} &\cdots & \sqrt{n_{2^{k-1}}} & \sqrt{m-n_{2^{k-1}}}&\sqrt{m- n_{2^{k-1}-1}}&
\cdots & \sqrt{m-n_1}\\
\sqrt{m-n_1}&\sqrt{ m-n_2}&\cdots& \sqrt{m-n_{2^{k-1}}} &\sqrt{n_{2^{k-1}}} & 
\sqrt{n_{2^{k-1}-1} }& 
\cdots & \sqrt{n_1}
\end{bmatrix}\]

Then we write the first row $2^k$ times and below it the second row
$2^k$ times.  We now need to assign {\it signs} of $\pm$ to this $2^{k+1}\times 2^k$
matrix so that the columns are orthogonal.  This is easy to do.  We use a variation
of Hadamard matrices.  i.e.
\[ A_1 = 
\begin{bmatrix}
-1 & 1&1&1\\
1 & -1&1&1\\
1&1&-1&1\\
1&1&1&-1
\end{bmatrix}\]

And by induction,
\[ A_{n+1}=\begin{bmatrix}
A_n&A_n\\
A_n & -A_n
\end{bmatrix}\]

We use this twice,  i.e.  for each set of $2^k$ copies of the two rows.
With this class of signs, we get a $2^{k+1}\times 2^k$ matrix with orthogonal columns
and square row sums 
\begin{equation}\label{Eq1} 
\sum_{i=1}^{2^{k-1}} n_i + \sum_{i=1}^{2^{k-1}}(m-n_{i}) = 2^{k-1}m,
\end{equation}
and square sums of the columns is just $2^km$.
So after normalization, this is a unit norm 2-tight frame.  

\vskip12pt
 Our examples above do this.  Also, the following works - for example:
\vskip12pt
\[32\ 31\ 30\ 29 \ \ \  \ \ \ \ \  25\ 24\ 23\ 22\ \ \   18\ 17\ 16 \ 15\ \ \    11\ 10\   9  \   8\]

\[   \   8\  \  9\  \ 10\ 11 \ \ \  \ \ \ \ \  15\ 16\ 17\ 18 \ \ \  \ \   22\ 23\ 24\ 25 \ \ \    29\ 30\ 31 \ 32  \]  

\vskip12pt

\noindent {\bf Further generalizations}:  

\begin{remark}
Note that the terms may not all be decreasing above.   This would require that
\[ n_{2^{k-1}} > m-n_{2^{k-1}}.\]
That is,
\[ m < 2n_{2^{k-1}}.\]
\end{remark}

\begin{remark}
Note that we do not really need the terms to be decreasing above.  All we really need
is for the rows to sum to $2^{k-1}m$ and the columns to sum to {\it m}.
For example,
\[ \begin{bmatrix}
20&24&19&25&1&7&2&6\\
6&2&7&1&25&14&24&20
\end{bmatrix}\]
\end{remark}

\begin{remark}  Note the we do not need the numbers to be distinct
above.  We can write one of them as many times as we like as long as we satisfy
Equation \ref{Eq1}.
\end{remark}

\begin{remark}  Note that this gives another solution to the paper of Casazza, 
Fickus, Mixon and Tremain \cite{CF} where they used variations of the DFT to construct
non-2-pavable matrices.  A careful choice of the numbers above putting most of the
weight in the first $n-1$ positions then works in their proof.
\end{remark}

Here is another {\it outline} for constructing 2-tight frames.  Fix $a,b,c,d$
and let $m=a+b+c+d$.

\[ \begin{bmatrix}
a&b&c&d&m-d&m-c&m-b&m-a\\
m-1&m-b&m-c&m-d&d&c&b&a
\end{bmatrix}\]

Then, the sum of the column numbers is $a+b+c+d$ and we have 8 of these
so the total sum after we turn this into a matrix is $8m$ while the sum of the rows
is $4m$.   So, after normalization, this is a unit norm 2-tight frame.

\section{General Tight Frames}

We will generalize one of our earlier examples.

\begin{theorem}
Let $\{A_i\}_{i=1}^n$ be $r\times k_i$ matrices with orthogonal
columns.  Let
\[ \left ( a_{ij} \right )_{i=1,j=1}^{\ m\  ,\ n},\mbox{ be a matrix
with orthogonal columns}.\]
Then
\[ \begin{bmatrix}
a_{11}A_1& a_{12}A_2 & \cdots &a_{1n}A_n\\
a_{21}A_1&a_{22}A_2& \cdots &a_{2n}A_n\\
\vdots & \vdots& \vdots& \vdots\\
a_{m1}A_1&a_{m2}A_2 & \cdots & a_{m n}A_n
\end{bmatrix}\]
is a $mr\times L$ matrix with orthogonal columns where
\[ L = \sum_{i=1}^n k_i.\]
\end{theorem}

\begin{proof}
Obvious.
\end{proof}

This leads to:

\begin{theorem}
Let $A_1,A_2,\ldots,A_n$ be $km \times m$ matrices with unit norm rows and
orthogonal columns which square sum to $km$.  i.e.  These are 
unit norm $k$-tight frames.
Let
\[ \left ( a_{ij} \right )_{i,j=1}^n,\mbox{ be an orthonormal matrix}.\]
Then the following matrix is a $kmn\times mn$ unit norm tight frame.
\[ \begin{bmatrix}
a_{11}A_1& a_{12}A_2 & \cdots &a_{1n}A_n\\
a_{21}A_1&a_{22}A_2& \cdots &a_{2n}A_n\\
\vdots & \vdots& \vdots& \vdots\\
a_{n1}A_1&a_{2n}A_2 & \cdots & a_{nn}A_n
\end{bmatrix}\]
\end{theorem}

\begin{proof}
Obvious.
\end{proof}

\begin{remark}
The same theorems hold with {\it rows} replaced by {\it columns}.  
\end{remark}

\section{Unbiased Bases}

Here is a way to get three unbiased bases in $\RR^4$.  One will be the unit vectors.
The other two are below where $+,-$ means $\pm1$:
\[  \frac{1}{2}
\begin{bmatrix}
-&+&+&+\\
+&-&+&+\\
+&+&-&+\\
+&+&+&-\\
\end{bmatrix}\]
And the last basis:
\[ \frac{1}{2}\begin{bmatrix}
+&+&+&+\\
+&+&-&-\\
+&-&-&-\\
+&-&+&-
\end{bmatrix}\]


\begin{thebibliography}{WW}

\bibitem{CE}  P.G. Casazza, D. Edidin, D. Kalra and V. Paulsen, 
{\it Projections and the Kadison-Singer Problem}, 
Operators and Matrices,  Vol. 1, No. 3 (2007) 391-408.

\bibitem{CF}  \item P.G. Casazza, M. Fickus, D. Mixon and J.C. Tremain, {\it Concrete constructions
of non-pavable projections}, preprint.


\bibitem{CT}  P.G. Casazza and J.C. Tremain, {\it The Kadison-Singer Problem in 
mathematics and engineering}, Proceedings of the National Academy of Sciences, 
{\bf Vol. 103} No. 7 (2006) 2032-2039.

\bibitem{CT1}  P.G. Casazza, M. Fickus, J.C. Tremain and E. Weber, 
{\it The Kadison-Singer Problem in mathematics and engineering---
A detailed account}, Contemporary Math, {\bf 414}, Operator theory, 
operator algebras and applications, D. Han, P.E.T. Jorgensen and D.R. 
Larson Eds. (2006) 297-356.

\end{thebibliography}
\end{document}